\documentclass[11pt,reqno,oneside]{amsart}
\usepackage{graphics}

\newtheorem{theorem}{Theorem}[section]
\newtheorem{lem}[theorem]{Lemma}
\newtheorem{pro}[theorem]{Proposition}
\theoremstyle{definition}

\theoremstyle{remark}
\newtheorem{rem}[theorem]{Remark}
\numberwithin{equation}{section}
\def\R{\mathbb{R}}
\def\O{\Omega}

\def\e{\epsilon}

 \DeclareMathOperator{\diam}{diam}

\def\RR{{\rm I}\!{\rm R}}

\def\1{(1+|\xi|)}

\def\O{\Omega}
\def\tt{\tilde\theta}

\begin{document}
\title []{On the existence of bounded solutions for a nonlinear elliptic
system}

\author{Ricardo G. Dur\'an}
\address{Departamento de Matem\'atica, Facultad de Ciencias Exactas y Naturales,
Universidad de Buenos Aires, 1428 Buenos Aires, Argentina} \email{rduran@dm.uba.ar}
\author{Marcela Sanmartino}
\address{Departamento de Matem\'atica, Facultad de Ciencias Exactas,
Universidad Nacional de La Plata, 1900 La Plata (Buenos Aires), Argentina }

\email{tatu@mate.unlp.edu.ar}
\author{ Marisa Toschi}
\address{Departamento de Matem\'atica, Facultad de Ciencias Exactas,
Universidad Nacional de La Plata, 1900 La Plata (Buenos Aires), Argentina }

\email{mtoschi@mate.unlp.edu.ar}

\thanks{Supported by ANPCyT (PICT 01307), by Universidad de
Buenos Aires (grant X070), by Universidad Nacional de La Plata (grant X500), 
and by CONICET (PIP 11220090100625).
The first author is a member of CONICET, Argentina.}



\begin{abstract}

This work deals with the system $(-\Delta)^m u= a(x)\, v^p$, $(-\Delta)^m v=b(x)\, u^q$ with
Dirichlet boundary condition in a domain $\Omega\subset\RR^n$, where $\Omega$ is a ball if $n\ge 3$
or a smooth perturbation of a ball when $n=2$.

We prove that, under appropriate conditions on the parameters ($a,b,p,q,m,n$), any non-negative solution
$(u,v)$ of the system is bounded by a constant independent of $(u,v)$. Moreover, we prove
that the conditions are sharp in the sense that, up to some border case, the relation
on the parameters are also necessary.

The case $m=1$ was considered by Souplet in \cite{PS}. Our paper
generalize to $m\ge 1$ the results of that paper. 
\end{abstract}

\maketitle

\section{Introduction}

In this paper we consider the nonlinear problem

\begin{eqnarray}
\label{1.1}
\left\{\begin{array}{ccc}
(-\Delta)^m u= a(x)\, v^p &\mbox{ in }\Omega\\
(-\Delta)^m v=b(x)\, u^q&\mbox{ in }\Omega\\
\left(\frac{\partial}{\partial \nu}\right)^{j}u=\left(\frac{\partial}{\partial \nu}\right)^{j}v=0
&\mbox{ on }\partial\Omega
& 0\leq j\leq m-1,\end{array}\right.
\end{eqnarray}
where $\Omega$ is the unit ball, namely, $\Omega=B=\{x\in\R^n\,:\,|x|< 1\}$ when $n\geq 3$,
and $B$ or some perturbations of $B$ for the case $n=2$ (see \cite{DS2} for details of this perturbation),
$\frac{\partial}{\partial \nu} $ is the normal derivative, $p,\, q > 0$, $pq > 1$,
and $a,\, b$ are nonnegative bounded functions. Let us remark that the restriction on the
domains is due to the fact that we will use that the Green function of the corresponding
linear problem is positive.

For the case $m=1$, a priori bounds for non-negative solutions of (\ref{1.1}) in a $C^2$
bounded domain $\Omega$ were
obtained by P. Souplet in \cite{PS}. To recall the results in that paper we introduce
$$
\alpha= \frac{2(p+1)}{pq-1}\ \mbox{     and     }\  \beta=\frac{2(q+1)}{pq-1}.
$$
Souplet proved that, if $\max\{\alpha,\beta\}>n-1$, then
\begin{equation}
\label{cota1}
\|u\|_{L^\infty(\O)} ,\, \|v\|_{L^\infty(\O)}\leq C,
\end{equation}
where the constant $C$ depends only on $p,q,a,b$, and $\Omega$.

Moreover, he proved that the result is sharp in the sense that, if
$\max\{\alpha,\beta\}<n-1$, then there exist a non-negative solutions of (\ref{1.1})
which are not bounded.

Our goal is to obtain similar results for non-negative solutions of (\ref{1.1})
for a general $m$.

A key tool used in \cite{PS} are some weighted a priori
estimates for the associated linear problem given by
\begin{eqnarray*}
\left\{\begin{array}{ccc}
-\Delta u= f &\mbox{ in }\Omega\\
u=0 &\mbox{ on }\partial\Omega.
\end{array}\right.
\end{eqnarray*}

Then, in order to generalize the a priori estimates for the case $m\ge 2$
we will need to extend the weighted estimates to higher order linear
problems. Non trivial technical modifications are needed to prove those estimates.
Moreover, since we need to use positivity of the Green function, we have to restrict
the domain $\Omega$ as mentioned above. Indeed, for $m \ge 2$ and general regions
the Green function is not necessarily positive.

\section{Weighted a priori estimates for the linear problem}

We will denote by $d(x)$ the distance from $x$
to the boundary of $\Omega$ and we will work with the Banach
space $L^p_{d^m}(\Omega)$ where the norm is given by
$$
\|u\|_{L_{d^m}^{p}(\O)}
= \left(\int_\O |u|^p\, d^m\, dx\right)^{1/p}
$$
for $1\le p<\infty$ and  $\|u\|_{L_{d^m}^\infty(\O)}:=\|u\|_{L^\infty(\O)}$.

In our arguments we will use some results given in \cite{D-S}
for the linear problem
\begin{eqnarray}
\label{2.1}
\left\{\begin{array}{ccc}
(-\Delta)^m u= f &\mbox{ in }\Omega\\
\left(\frac{\partial}{\partial \nu}\right)^{j}u=0 &\mbox{ on
}\partial\Omega & 0\leq j\leq m-1.\end{array}\right.
\end{eqnarray}
We recall those results in the following lemma.
Let us remark that these results,
and consequently our proposition below,
are valid in more general domains than those considered here.
Indeed, the hypotheses used
 are
that $\O$ is a bounded domain with
$C^{6m+4}$ boundary for $n=2$ and
$C^{5m+2}$ boundary for $n> 2$.

\begin{lem}
\label{propDS}
Let $u\in C^{2m}(\overline{\O})$ and $f\in C(\overline\O)$ satisfy (\ref{2.1}).

$\bullet$ If $2m>n$, then there exists $C>0$ such that for all $\theta\in [0,1]$
 $$ \|u\, d^{-m+\theta n}\|_{L^\infty(\O)}
 \leq C\,  \|f\, d^{m-(1-\theta) n}\|_{L^1(\O)}.$$

 $\bullet$ Let $1\leq p\leq q\leq \infty$. If $\frac{1}{p}-\frac{1}{q}<\min \{\frac{2m}{n}, 1\}$,
 then taking $ \alpha\in(\frac{1}{p}-\frac{1}{q},\min \{\frac{2m}{n}, 1\}]$ there exists $C>0$
 such that for all $\theta\in [0,1]$
  $$ \|u\, d^{-m+\theta n\alpha}\|_{L^q(\O)}
 \leq C\,  \|f\, d^{m-(1-\theta) n\alpha}\|_{L^p(\O)}.
 $$
 \end{lem}

\begin{proof}
See  Proposition 4.2 in \cite{D-S}.
\end{proof}

We have the following a priori estimates for solutions of
problem (\ref{2.1}).

\begin{pro}
\label{pro2.1}
Let $1\leq p\leq q\leq \infty$. Let $f\in L^p_{d^m}(\Omega)$ and
let $u$ be a weak solution of (\ref{2.1}).

We have
\begin{enumerate}
\item if $n\leq m$, then $u\in L^\infty(\Omega)$ and there exists $C>0$ such that
\begin{equation*}\label{nleqm}
\|u\|_{L^\infty(\Omega)}\leq C\, \|f\|_{L^1_{d^m}(\Omega)}
\end{equation*}
\item if $\frac{1}{p}-\frac{1}{q}< \frac{2m}{n+m}$, then
 $u\in L^p_{d^m}(\Omega)$ and there exists $C>0$ such that
\begin{equation*}\label{2m}
\|u\|_{L^q_{d^m}(\Omega)}\leq C\, \|f\|_{L^p_{d^m}(\Omega)}.
\end{equation*}
\end{enumerate}
\end{pro}
\begin{proof}
From Lemma \ref{propDS} we have that, for $2m>n$ and $\theta\in [0,1]$,

\begin{equation}
\label{16ds}
\|u\, d^{-m+\theta n}\|_{L^\infty(\O)}\leq C\, \| f\,
d^{m-(1-\theta)n}\|_{L^1(\O)}.
\end{equation}

Then taking $\theta=1$ and using that $-m+n<0$ and $d(x)\leq
\diam(\Omega)$ we obtain
$$
\|u\|_{L^\infty(\O)} \leq C\, \|u\, d^{-m+n}\|_{L^\infty(\O)}\leq C\, \| f\,
d^{m}\|_{L^1(\O)}
$$
and so $(1)$ is proved.

On the other hand, using again Lemma \ref{propDS}, we have that, if there exists
$\alpha\in(\frac1p-\frac1q, \min \{1, \frac{2m}{n}\}]$ and
$\theta\in
[0,1]$ such that
\begin{eqnarray}
\label{potenciam}
\left\{\begin{array}{ccc}
-m+\theta\, n\, \alpha=\frac{m}{q}\\
m-(1-\theta)\, n\, \alpha=\frac{m}{p}
\end{array}\right.
\end{eqnarray}
 we obtain
$$\|u\|_{L^q_{d^m}(\O)}\leq C\, \|f\|_{L^p_{d^{m}}(\O)}$$
for $\frac1p-\frac1q<\min \{1, \frac{2m}{n}\}$.

Solving system (\ref{potenciam}) we obtain
$$
\alpha=(2+\frac{1}{q}-\frac{1}{p})\frac{m}{n}\ \mbox{\ and\ }
\ \theta=(\frac1q+1)\, (2-\frac1p+\frac1q)^{-1}.
$$

We are going to show that $\alpha$ and
$\theta$ satisfy the required conditions
if $\frac{2m-n}{m}\le \frac{1}{p}-\frac{1}{q}< \frac{2m}{n+m}$.

Since $1\le p$ we have $\theta\in [0,1]$. On the other
hand, from the definition of $\alpha$, it is easy to see
that the condition
$\frac{1}{p}-\frac{1}{q}<\alpha$ is equivalent
to $\frac{1}{p}-\frac{1}{q}< \frac{2m}{n+m}$, which is one of our hypothesis.

Finally we have to see that $\alpha\leq \min \{1, \frac{2m}{n}\}$.
Since $p\le q$ we have $\alpha\le\frac{2m}{n}$. Therefore, it only remains
to consider the case $\frac{2m}{n}> 1$. But $\alpha\le 1$ is
equivalent to $\frac{2m-n}{m}\le \frac{1}{p}-\frac{1}{q}$ and so
the proposition is proved under this restriction.

Suppose now that $\frac{1}{p}-\frac{1}{q} < \frac{2m-n}{m}$.
In this case, for $2m>n$, using again the first part of Lemma \ref{propDS},
for all $\tt\in [0,1]$ we have
$$
\|u\, d^{-m+\tt n}\|_{L^\infty(\O)}
\le C\, \|f\, d^{m-(1-\tt) n}\|_{L^1(\O)}.
$$
Moreover, if $\tt\le\frac{m}{nq}+\frac{m}{n}$, it follows that
$$
\|u\|_{L^q_{d^m}(\O)}\le \|u\, d^{-m+\tt n}\|_{L^\infty(\O)}.
$$
Analogously, if $1-\frac{m}{n}+\frac{m}{np}\le\tt$,
$$
 \|f\, d^{m-(1-\tt) n}\|_{L^1(\O)}
 \le C
 \|f\|_{L^p_{d^{m}}(\O)}.
$$
Therefore, if we can choose $\tt$ satisfying
$1-\frac{m}{n}+\frac{m}{np}\leq
\tt\leq \frac{m}{nq}+\frac{m}{n}$ we have
$$
\|u\|_{L^q_{d^m}(\O)}\leq C\, \|f\|_{L^p_{d^{m}}(\O)},
$$
but, such a $\tt$ exists because
$\frac{1}{p}-\frac{1}{q} \leq \frac{2m-n}{m}$
and the proposition is proved.
\hspace*{\fill} $\square$
\end{proof}

\begin{rem}
\label{remarkcita}
The condition in (2) is almost optimal, i.e., if $\frac{1}{p}-\frac{1}{q}> \frac{2m}{n+m}$ then
the a priori estimate does not hold in general. We postpone the proof of this
observation to the end of the paper because we will use the same technique as in the
proof of our second main theorem.
\end{rem}
In the proof of the following proposition we will denote with
$\lambda_{1,m}$ the first eigenvalue of the operator
$(-\Delta)^m$ and with $\phi_{1,m}>0$ a corresponding eigenfunction
normalized by $\int_\O \phi_{1,m}=1$.
We will use that there exist two positive constants $c_1$ and $c_2$ such that, in $\O$,
\begin{equation}
\label{cotasautofuncion}
c_1\, d^m\leq \phi_{1,m}
\le c_2\, d^m ,
\end{equation}
see \cite{CS}.

\begin{pro}
\label{proposition2.3}
Let $1\leq k < \frac{n+m}{n-m}$.
If $u$ is a 
solution of (\ref{2.1}) with
$f\in L^1_{d^m}(\Omega)$ and $f>0$, then there exists $C>0$ such that
\begin{equation}\label{2.3}
\|u\|_{L^k_{d^m}(\O)}\leq C\, \|u\|_{L^1_{d^m}(\O).}
\end{equation}
\end{pro}
\begin{proof}
Taking $p=1$  in the previous proposition we obtain for $1\leq k<
\frac{n+m}{n-m}$
$$\|u\|_{L^k_{d^m}(\O)}\leq C\, \|f\|_{L^1_{d^m}(\O)}.$$
Using integration by
parts and that $f>0$ we have
\begin{align*}
 \|f\|_{L^1_{d^m}}&= \int_\Omega (-\Delta)^m u\, d^m\, dx \le \int_\Omega (-\Delta)^m u\, \phi_{1,m}\, dx
  \\
 &\leq C\,  \int_\Omega u\, (-\Delta)^m \phi_{1,m}\, dx = C\, \lambda_{1,m}  \int_\Omega u\, \phi_{1,m}\, dx
 \\
 &\leq C\, \int_\Omega u\, d^m\,
 dx \leq C\, \|u\|_{L^1_{d^m}}.
  \end{align*}
\hspace*{\fill} $\square$
\end{proof}

\section{Main results}

We consider problem (\ref{1.1}) and define the exponents
$$
\alpha= \frac{2m(p+1)}{pq-1}\ \mbox{     and     }\  \beta=\frac{2m(q+1)}{pq-1}.
$$

Then, the natural extension of the results in \cite{PS} is given by the following

\begin{theorem}
\label{m}
If
\begin{equation}
\label{1.3}
\max(\alpha,\beta)>n-m,
\end{equation}
then, any non-negative solution of (\ref{1.1}) satisfies
\begin{equation}
\|u\|_{L^\infty(\O)} ,\, \|v\|_{L^\infty(\O)}\leq C
\end{equation}
where $C$ is a positive constant which depends only on $a$, $b$, $p$, $q$, $m$, and $\O$.
\end{theorem}

We also prove, in the following theorem, that condition (\ref{1.3}) is almost optimal.
We cannot say optimal because we do not know what happens in the case $\max(\alpha,\beta)= n-m$.

\begin{theorem}
\label{optima}
If
\begin{equation}\label{1.5}
 \max(\alpha,\beta)< n-m,
 \end{equation}
then, there exist nonnegative bounded functions $a$ and $b$, such that (\ref{1.1})
have some non-negative solution $(u,v)$, with $u$ and $v$ unbounded functions.
\end{theorem}

Once we have the results of the previous section, the proofs of both theorems
follows the lines of the case $m=1$ proved in \cite{PS}. A key point in
the arguments given in that paper are the estimates
\begin{equation}
\label{prop4.1}
\int _\Omega u\, \phi_{1,m}\, , \, \int_\Omega v\, \phi_{1,m}\le C.
\end{equation}
A straightforward extension of the arguments given in \cite{PS2}, to prove these estimates
in the case $m=1$, is not possible. Indeed, the proof given in that paper is based
on a lemma of \cite{BC} which uses the maximum principle in subsets of $\O$.
An analogous maximum principle is not valid in the case $m\ge 2$.
We will give a different proof of this lemma using pointwise estimates for the
Green function $G_m$ of problem (\ref{2.1}) given below. This is why we have to restrict $\Omega$ in order to
have that the Green function is positive, i.e., we assume
that $\Omega=B=\{x\in\R^n\,:\,|x|< 1\}$ when $n\geq 3$,
and $\O=B$ or some perturbations of $B$ for the case $n=2$ (see \cite{DS2} for details of this perturbation).
We have:
for $ 2m<n$,
\begin{equation}\label{g_1}
 G_m(x,y)\geq C\, |x-y|^{2m-n}\, \min\left\{1,\, \frac{ d(x)^m\, d(y)^m}{|x-y|^{2m}}\right\},
 \end{equation}
for $ 2m=n$,
\begin{equation}\label{g_2}
 G_m(x,y)\geq C\, \log\left(1 + \frac{ d(x)^m\, d(y)^m}{|x-y|^{2m}}\right)
 \geq C\,
\log\left(2 + \frac{ d(y)}{|x-y|}\right)\min\left\{1 , \frac{
d(x)^m\, d(y)^m}{|x-y|^{2m}}\right\},
 \end{equation}
and for $ 2m>n$,
\begin{equation}\label{g_3}
G_m(x,y)\geq C\, d(x)^{m-n/2}\, d(y)^{m-n/2}\, \min\left\{1,\, \frac{ d(x)^{n/2}\, d(y)^{n/2}}{|x-y|^{n}}\right\}.
 \end{equation}

The proofs of these estimates can be found in \cite{DS2} for the case
of $m=n =2$ and in \cite{GS} for the rest of the cases.

\begin{lem}
\label{x}
Assume $h\geq 0$, $h\in L^1_{d^m}(\Omega)$ and $v$ a solution
of
\begin{eqnarray}\label{xavier}
\left\{\begin{array}{ccc}
(-\Delta)^m v=h &\mbox{ in }\Omega\\
\left(\frac{\partial}{\partial \nu}\right)^{j}v=0 &\mbox{ on
}\partial\Omega & 0\leq j\leq m-1.\end{array}\right.
\end{eqnarray}
Then there exists $C>0$, depending only on $\O$ and $m$, such that for all $x\in \Omega$
\begin{equation}
\frac{v(x)}{d^m(x)}\geq C\, \int_\Omega h\, d^m.
\end{equation}
\end{lem}
\begin{proof}
By the representation formula
$$v(x)=\int_\Omega G_m(x,y)\, h(y)\, dy$$
it is enough to prove that
$$G_m(x,y)\geq C\, d(x)^m\, d(y)^m.$$

Consider, for example, the case $2m<n$ and suppose that
$\frac{ d(x)^m\, d(y)^m}{|x-y|^{2m}}\ge 1$. Then, it follows from
(\ref{g_1}), that
$$
G_m(x,y)\ge C\, |x-y|^{2m-n}\ge d(x)^{m-n/2} d(y)^{m-n/2}
\ge C d(x)^{m} d(y)^{m}
$$
where in the last step we have used that $\O$ is bounded.
On the other hand, if the minimum on the right hand side
of (\ref{g_1}) is attained in $\frac{ d(x)^m\, d(y)^m}{|x-y|^{2m}}$
we have
$$
G_m(x,y)\ge C\, |x-y|^{-n}d(x)^{m} d(y)^{m}
\ge C d(x)^{m} d(y)^{m}.
$$
The proofs for the cases $2m=n$ and $2m>n$ are analogous, using now (\ref{g_2})
and (\ref{g_3}) respectively.
\hspace*{\fill} $\square$
\end{proof}

\bigskip
{\bf Proof of Theorem \ref{m}:}

\noindent Step 1: From (\ref{prop4.1}) and (\ref{cotasautofuncion}) it follows
immediately,

\begin{equation}
\|u\|_{L^1_{d^m}}+ \|v\|_{L^1_{d^m}}\leq C,
\end{equation}
and therefore, for $n\leq m$,
\begin{equation*}
 \|u\|_{L^\infty(\O)} ,\, \|v\|_{L^\infty(\O)}\leq C
\end{equation*}
is an immediate consequence of $(1)$ in Proposition
\ref{pro2.1}.

On the other hand, if $ n>m$, it follows from
Proposition \ref{proposition2.3}, that
\begin{equation}\label{k}
\|u\|_{L^k_{d^m}}+ \|v\|_{L^k_{d^m}}\leq C(k)
\end{equation}
for $1\leq k< \frac{n+m}{n-m}.$

Clearly we may assume $q\geq p$ and $\beta>n-m$. Then, it is easy to check
that $p<\frac{n+m}{n-m}$ and so, there exists
some $k$ such that
\begin{equation}\label{3.6}
 k\geq p\ \  \mbox{ and }\ \  k\geq \frac{n+m}{n-m}-\epsilon,
  \end{equation}
with $\epsilon$ to be chosen below, for which (\ref{k}) holds.

\noindent Step 2: Assume now that we can choose $k_1\in (k,\, \infty]$ such that
\begin{equation}
\label{3.7}
\frac{1}{k_1}>\frac{p}{k}-\frac{2m}{n+m}.
\end{equation}
Then, using Proposition \ref{pro2.1} we have
\begin{equation}
\label{3.8}
\|u\|_{L^{k_1}_{d^m}}\leq C\, \|(-\Delta)^m u\|_{L^{k/p}_{d^m}}\leq C\,
\|v^p\|_{L^{k/p}_{d^m}}=C\, \|v\|^p_{L^{k}_{d^m}},
\end{equation}
which is finite because $1\leq k<\frac{n+m}{n-m}$.

Observe that, if $k>\frac{(n+m) pq}{2m(q+1)}$, we can take
$k_1>\frac{(n+m)q}{2m}$ satisfying
(\ref{3.7}).

%
%

\noindent Step 3: Assume
 \begin{equation}\label{3.11} k_1>q
\end{equation}
and let $k_2\in (k_1,\, \infty]$ be such that
\begin{equation}
\label{3.12}
\frac{1}{k_2}>\frac{q}{k_1}-\frac{2m}{n+m}.
\end{equation}
From Proposition \ref{pro2.1} we have
\begin{equation*}
\|v\|_{L^{k_2}_{d^m}}\leq C\, \|(-\Delta)^m v\|_{L^{{k_1}/q}_{d^m}}\leq C\,
\|u^q\|_{L^{{k_1}/q}_{d^m}}=C\, \|u\|^q_{L^{k_1}_{d^m}}
\end{equation*}
which is finite by step 2.

\noindent Step 4: We can see that conditions
 (\ref{3.7}), (\ref{3.11}), (\ref{3.12}) and
 $\min\{k_1,k_2\}>\frac{k}{\rho}$ for $\rho \in (0,1)$,
 to be chosen below,
are equivalent to
 \begin{equation}\label{3.14}
A:=\frac{p}{k}-\frac{2m}{n+m}<\frac{1}{k_1}<\min\left\{\frac{\rho}{k},\frac{1}{q}\right\}
\end{equation}
and
\begin{equation}\label{3.15}
\frac{q}{k_1}-\frac{2m}{n+m}<\frac{1}{k_2}<\frac{\rho}{k}.
\end{equation}
Observe now that, if
\begin{equation}
\label{3.16}
k\le\frac{(n+m)\, pq}{2m(q+1)},
\end{equation}
we have $A>0$. Therefore
(\ref{3.14}) can be solved for  $k_1\in [1, +\infty)$ and with
$\frac{1}{k_1}$ arbitrarily closed to $A$ whenever
\begin{equation}
\label{3.17}
\frac{p-\rho}{k}<\frac{2m}{n+m}
\end{equation}
and
\begin{equation}\label{3.18}
\frac{p}{k}-\frac{2m}{n+m}<\frac{1}{q}.
\end{equation}
But, (\ref{3.17}) holds if $\rho$ satisfies
 \begin{equation}\label{3.19}
\frac{n-m}{n+m}\, p<\rho<1,
 \end{equation}
and such a $\rho$ exists because $p<\frac{n+m}{n-m}$.

On the other hand, since $\beta=\frac{2m\, (q+1)}{pq-1}>n-m$, we have
$\frac{1}{q} >\frac{p\, (n-m)}{n+m} -\frac{2m}{n+m}$. Then,
since $k<\frac{n-m}{n+m}$ we can choose $\e$ such that
(\ref{3.18}) holds.

Let us now see that condition (\ref{3.15}) can be fulfilled. Indeed, it is enough to see that
all our parameters can be chosen such that
\begin{equation}\label{3.20}
\frac{q}{k_1}-\frac{2m}{n+m}<\frac{\rho}{k}.
\end{equation}

Taking $\frac{1}{k_1}$ in (\ref{3.14}) closed enough to A
we have that (\ref{3.20}) is equivalent to
\begin{equation}\label{3.21}
\rho>1-\eta,
\end{equation}
where $\eta:=\frac{2m}{n+m}\, (q+1)\, k-(pq-1)$.

Indeed, if
$\frac{1}{k_1}$ is closed to $A=\frac{p}{k}- \frac{2m}{n+m}$, then
$\frac{q}{k_1}-\frac{2m}{n+m}$ is closed to
$\frac{qp}{k}-
\frac{2mq}{n+m}-\frac{2m}{n+m}$.

%

Now, $\rho<1$ is equivalent to
\begin{equation}
\label{3.100}
k>\frac{n+m}{\beta},
\end{equation}
but since
$\beta>n-m$ it is possible to take $\e$ small enough in (\ref{3.6})
such that \eqref{3.100} is satisfied.

Finally we can take
$\rho\in (0,1)$ closed enough to one such that que (\ref{3.19}) and (\ref{3.21}) hold.

Step 5: It follows from step 4 that if (\ref{k}) holds for
some $k$ satisfying (\ref{3.6}) and (\ref{3.16}), then (\ref{k})
is true with $k/\rho$ (as a consequence of (\ref{3.14}) and (\ref{3.15})).

Iterating the procedure we can reach, after a finite number of steps,
some value $\bar{k}>
\frac{(n+m) pq}{2m(q+1)}$. Then, it follows from the comment at the end of step 2 that
there exists $\bar{k}_1>\frac{(n+m) q}{2m}\geq
\frac{(n+m) p}{2m}$ such that $\|u\|_{L^{\bar{k}_1}_{d^m}}\leq C$.

Taking now $k_1=\bar{k}_1$, we can take  $k_2=\infty$ in step 3 to conclude that
$\|v\|_{L^\infty(\O)} \leq C$. Analogously, by step 2 we obtain $\|u\|_{L^\infty(\O)} \leq C$. \hspace*{\fill} $\square$

\section{Existence of singular solutions.}

In order to prove Theorem \ref{optima}
we follow the ideas of \cite{PS}. First we will construct a function
$f\in L^1_{d^m}(\O)$ such that the corresponding solution of the linear problem (\ref{2.1})
is not bounded.

Recall that our domain $\O$ is a ball when $n\geq 3$, and smooth
perturbations of a ball in the case $n=2$. In any case, given $x_0\in
\partial\Omega$, there exist
$r>0$ and a revolution cone $\Sigma_1$ with vertex $x_0$ such that
$\Sigma:=\Sigma_1\cap B_{2r}(x_0)\subset \Omega$.
Now, for $0<\alpha<n-m$ we define
$$
f(x)=|x-x_0|^{-(\alpha+2m)}\chi_{\Sigma},
$$
where $\chi_{\Sigma}$ denotes the characteristic function of $\Sigma$.
Then, it is easy to see that $f\in L^1_{d^m}(\Omega)$.

Let  $u>0$  be the solution of (\ref{2.1}) with $f$ as
right-hand side. Then, we have
$$
u(x)=\int_\O G_m(x,y)\, |y-x_0|^{-(\alpha+2m)}\chi_{\Sigma}(y)\, dy.
$$
Using this representation formula
together with the estimates of the Green function
$(\ref{g_1})$, $(\ref{g_2})$ and $(\ref{g_3})$
it is not difficult to see that, for $x\in\O$,

\begin{equation}\label{unoacotada}
u(x) \geq C\, |x-x_0|^{-\alpha}\chi_{\Sigma}(x).
\end{equation}

\  \\{\bf Proof of Theorem \ref{optima}}:
Recall that $\alpha =\frac{2m(p+1)}{pq-1}$
and $\beta=\frac{2m(q+1)}{pq-1}$, and we are assuming $0<\alpha, \beta< n-m$.
We define
$$
\phi(x)= |x-x_0|^{-(\alpha+2m)}\, \chi_{\Sigma}(x)
\quad\mbox{and}
\quad\psi(x)=
|x-x_0|^{-(\beta+2m)}\, \chi_{\Sigma}(x).
$$
Let $u$ and $v$ be non-negative and such that
\begin{eqnarray*}
\left\{\begin{array}{ccc}
(-\Delta)^m u= \phi &\mbox{ in }\Omega\\
(-\Delta)^m v=\psi&\mbox{ in }\Omega\\
\left(\frac{\partial}{\partial
\nu}\right)^{j}u=\left(\frac{\partial}{\partial \nu}\right)^{j}v=0
&\mbox{ on }\partial\Omega & 0\leq j\leq m-1.\end{array}\right.
\end{eqnarray*}
Then, it follows from (\ref{unoacotada}) that
$u\notin L^\infty(\O)$, $v\notin L^\infty(\O)$,
$$
v(x)^p\geq \left(C\, |x-x_0|^{-\beta}\chi_\Sigma(x)\right)^p=C\,
|x-x_0|^{-(\alpha+2m)}\, \chi_\Sigma(x)= C\, \phi(x)
$$
and
$$
u(x)^q\geq \left(C\, |x-x_0|^{-\alpha}\chi_\Sigma(x)\right)^q=C\,
|x-x_0|^{-(\beta+2m)}\, \chi_\Sigma(x)= C\, \psi(x).
$$
Therefore, defining
$a=\phi/v^p$ and $b=\psi/u^q$ we have that $a$ and $b$ are nonnegative
bounded functions, and $(u,v)$ solves
$$
(-\Delta)^m
u= a(x)\, v^p \mbox{\ \  and\ \  }(-\Delta)^m v= b(x)\, u^q.
$$
\hspace*{\fill} $\square$

\medskip

We end the paper by proving the observation given in Remark \ref{remarkcita}
concerning the optimality  of condition
$(2)$ in Proposition \ref{pro2.1}.

\begin{pro}
 Assume $1\leq p\leq q\leq\infty$ and
 $\frac{1}{p}-\frac{1}{q}>\frac{2m}{n-m}$.  Then there exists
 $f\in L^p_{d^m}(\O)$ such that $u\notin L^q_{d^m}(\O)$, where
 $u$ is the unique solution of $(\ref{2.1}).$
\end{pro}

\begin{proof}
Let  $0<\alpha<n-m$ and we define, as above,
$f(x)=|x-x_0|^{-(\alpha+2m)}\chi_{\Sigma}(x).$ Then we have
$$
\|f\|^p_{L^p_{d^m}(\O)}=\int_\Sigma |x-x_0|^{-(\alpha +2m)p}\,
d(x)^m\, dx\le \int_\Sigma |x-x_0|^{-(\alpha +2m)p +m}\, dx,
$$
and then, since $p<\frac{n+m}{\alpha +2m}$, $f\in L^p_{d^m}(\O)$.

But, for $x\in \Sigma$
there exists a positive constant $C$ such that
$d(x)\ge C |x-x_0|$, and therefore, it follows from (\ref{unoacotada}) that for
$q\geq
 \frac{n+m}{\alpha}$,
$u\notin
L^q_{d^m}(\O)$.
To conclude the proof we observe that,
since $\frac{1}{p}-\frac{1}{q}>\frac{2m}{n-m}$, we can choose
$\alpha\in (0, n-m)$ such that
$\frac{n+m}{q}<\alpha<\frac{n+m}{p-2m}$.
\hspace*{\fill} $\square$
\end{proof}

\medskip

Finally let us mention that, to our knowledge, it is not known what
happens in general in the limit case $\frac{1}{p}-\frac{1}{q}=\frac{2m}{n-m}$.
In the case $p>m+1$ we have proved
in \cite{DST2} that
\begin{equation*}
\|u\|_{L^q_{d^m}(\Omega)}\leq C\, \|f\|_{L^p_{d^m}(\Omega)}.
\end{equation*}

\end{document}